\documentclass[12pt]{amsart}
\usepackage{amsfonts}
\usepackage{amsfonts,latexsym,rawfonts,amsmath,amssymb,amsthm}
\usepackage[plainpages=false]{hyperref}

\usepackage{graphicx}

\RequirePackage{color}

\numberwithin{equation}{section}

\newcommand{\beq}{\begin{equation}}
\newcommand{\eeq}{\end{equation}}
\newcommand{\beqs}{\begin{eqnarray*}}
\newcommand{\eeqs}{\end{eqnarray*}}
\newcommand{\beqn}{\begin{eqnarray}}
\newcommand{\eeqn}{\end{eqnarray}}
\newcommand{\beqa}{\begin{array}}
\newcommand{\eeqa}{\end{array}}

\newtheorem{prop}{Proposition}[section]
\newtheorem{thm}[prop]{Theorem}
\newtheorem{lem}[prop]{Lemma}

\newtheorem{cor}[prop]{Corollary}
\newtheorem{rem}[prop]{Remark}

\renewcommand{\div}{\mbox{div}\,}
\renewcommand{\det}{\operatorname{det}}

\allowdisplaybreaks
\title{Interior H\"older regularity of the linearized Monge-Amp\`ere equation}
\author{Ling Wang}

\address{School of Mathematical Sciences, Peking
University, Beijing 100871, China.}
\email{lingwang@stu.pku.edu.cn; lwmath@foxmail.com}

\thanks {This research is partially supported by  National Key R$\&$D Program of China SQ2020YFA0712800, 2023YFA009900 and NSFC  Grant 11822101.}

\begin{document}

\subjclass{35B65, 35J96, 35J75, 35J70.}

\keywords{interior H\"older regularity, linearized Monge-Amp\`ere equations, De Giorgi-Nash-Moser's theory, partial Legendre transforms, Moser-Trudinger inequality}

\begin{abstract}
    In this paper, we investigate the interior H\"older regularity of solutions to the linearized Monge-Amp\`ere equation. 
    In particular, we focus on the cases with singular right-hand side, which arise from the study of the semigeostrophic equation and singular Abreu equations. In the two-dimensional case, we give a new proof of the Caffarelli-Guti\'errez H\"older estimate (\textit{Amer. J. Math.} \textbf{119} (1997), no.\,2, 423-465) and the result of Le (\textit{Comm. Math. Phys.} \textbf{360} (2018), no.\,1, 271-305) for the  linearized Monge-Amp\`ere equation with singular right-hand side term in divergence form. The main new ingredient in the proof contains the application of the partial Legendre transform to the linearized Monge-Amp\`ere equation. Building on this idea, we also establish a new Moser-Trudinger type inequality in dimension two. In higher dimensions, we derive the interior H\"older estimate under certain integrability assumptions on the coefficients using De Giorgi's iteration.
  \end{abstract}

\maketitle

\section{Introduction}
\vskip 12pt
In this paper, we investigate the interior H\"older regularity for solutions to the inhomogeneous linearized Monge-Amp\`ere equation
\begin{equation}\label{eq:s-LMA-div}
    \sum_{i,j=1}^nD_j\left(\Phi^{ij}D_{i}u\right)=\div F+f
\end{equation}
in a bounded convex domain $\Omega\subset\mathbb R^n$ $(n\geq 2)$, where $\Phi=\left(\Phi^{ij}\right)$ is the cofactor matrix of the Hessian matrix of a convex function $\phi\in C^2(\Omega)$, $F:\Omega\to\mathbb R^n$ is a vector field, and $f:\Omega\to\mathbb R$ is a function. Since $\Phi$ is divergence free, i.e. $\displaystyle \sum_{j=1}^nD_j\Phi^{ij}=0$, for all $i=1,2,\cdots,n$, we can rewrite \eqref{eq:s-LMA-div} in the non-divergence form as follows:
\begin{equation}\label{eq:s-LMA}
    \sum_{i,j=1}^n\Phi^{ij}D_{ij}u=\div F + f.
\end{equation}

When the equation is uniformly elliptic, it is well known that the Harnack and H\"older estimates of the solution are established in the classical De Giorgi-Nash-Moser theory for equations of divergence form  \cite{De,Na,Mo}, and
in the Krylov-Safonov theory for general equations of non-divergence form \cite{KS}. 
The main interest on equation \eqref{eq:s-LMA-div} or \eqref{eq:s-LMA} lies in the lack of uniform ellipticity.
In a celebrated work,  Caffarelli-Guti\'errez obtained the Harnack inequality and the H\"older estimate for equation \eqref{eq:s-LMA-div} or \eqref{eq:s-LMA} with $F\equiv 0$ and $f\equiv 0$, under the $\mathcal{A}_\infty$ condition  \cite{CG}. In particular, this condition is satisfied if \begin{equation}\label{eq:det-bd}
    0<\lambda\leq\det  D^2\phi\leq\Lambda\quad \text{in }\Omega.
\end{equation}
The H\"older estimate for the inhomogeneous equation ($F=0$ and $f\neq 0$), as well as higher order estimates and the boundary regularity, were later established by \cite{TrW, GN1,GN2,LN1,LS}  under certain assumptions on $f$. For further extensions and related work, one can refer to \cite{Le1,Le5,LN2,KLWZ} and the references therein.

When $F\neq 0$, the equation arises from the study of semigeostrophic equations \cite{ACDF,Le,Lo} and singular Abreu equations \cite{KLWZ,Le4,LZ} in the study of convex functionals with a convexity constraint related to the Rochet-Chon\'e model for the monopolist problem in economics. So far, very little is known about the regularity of  \eqref{eq:s-LMA-div} when $F\neq 0$. We focus mainly on the linearized Monge-Amp\`ere equation under the condition \eqref{eq:det-bd}.
Loeper \cite{Lo} obtained  the interior H\"older regularity for \eqref{eq:s-LMA-div} under the stronger assumption that $\det D^2\phi$ is sufficiently close to a positive constant, using the $W^{2,p}$ estimate of the Monge-Amp\`ere equation and a result derived in \cite{Tr}. Later, Le \cite{Le} showed the same result when $n=2$, only assuming \eqref{eq:det-bd}. The main ingredient used in \cite{Le} is the $W^{2,1+\varepsilon}$-estimate of the Monge-Amp\`ere equation established by \cite{DFS,Sc}. Under certain integral bounds on the Hessian $D^2\phi$, Le \cite{LeBook} also extended the H\"older estimates to higher dimensions. More recently, Kim \cite{Kim} derived similar estimates for \eqref{eq:s-LMA-div} with drift terms under a similar condition as \cite{LeBook}. For the estimates to boundary case, see \cite{Le3}.

The results in \cite{Kim,Le,LeBook,Lo} used both the De Giorgi-Nash-Moser iteration and the Caffarelli-Guti\'errez estimate, corresponding to the divergence form and the non-divergence form of linearized Monge-Amp\`ere equations, respectively. Note that it has been pointed out that in general it is impossible to obtain the Caffarelli-Guti\'errez estimate by the De Giorgi-Nash-Moser iteration \cite[Remark 3.4]{TW}, which means that the celebrated theory of Caffarelli-Guti\'errez's is essential in their arguments. However, in this paper we find that there is a new proof of the theorem below without using the Caffarelli-Guti\'errez estimate in dimension two. 

\begin{thm}\label{thm:main}
    Assume $n=2$. Let $\phi\in C^2(\Omega)$ be a convex function satisfying \eqref{eq:det-bd}. Let $F:=(F^1(x),F^2(x)):\Omega\to\mathbb R^2$ be a bounded vector field and $f\in L^r(\Omega)$ for $r>1$. Given  $\Omega'\subset\subset\Omega$ and $p\in(0,+\infty)$, then for every solution $u$ to \eqref{eq:s-LMA-div}
    in $\Omega$, there is
    \begin{equation*}
        \|u\|_{C^{\gamma}(\Omega')}\leq C\left(\|u\|_{L^p(\Omega)}+\|F\|_{L^\infty\left(\Omega\right)}+\|f\|_{L^r(\Omega)}\right),
    \end{equation*}
    where constant $\gamma>0$ depending only on $\lambda$ and $\Lambda$, and constant $C>0$ depending only on $p$, $r$, $\lambda$, $\Lambda$, and $\operatorname{dist}(\Omega',\partial\Omega)$.
\end{thm}

Note that although we assume $\phi\in C^2(\Omega)$, the derived estimates are independent of the smoothness of $\phi$ and depend only on the structure constants.
Theorem \ref{thm:main} includes \cite[Theorem 1.3]{Le} and Caffarelli-Guti\'errez's estimate \cite{CG} in dimension two. Our main new idea in Theorem \ref{thm:main} is the use of the partial Legendre transform.
After the partial Legendre transform, \eqref{eq:s-LMA-div} becomes a linear uniformly elliptic equation in divergence form with singular right-hand side (see \eqref{eq:plt-lma}). Thus, the De Giorgi-Nash-Moser theory implies that the solution after transformation is H\"older continuous. Then transforming back to the original solution gives us the result. We still need the $W^{2,1+\varepsilon}$-estimate of the Monge-Amp\`ere equation to guarantee that the condition in De Giorgi-Nash-Moser's theory is satisfied.
The partial Legendre transform has been widely used in the study of the Monge-Amp\`ere equation \cite{DS, Fi, GP, Liu}, and it has also been used recently to study the Monge-Amp\`ere type fourth order equation \cite{LZ,WZ}.  However, we didn't find its use in the linearized Monge-Amp\`ere equation. Our proof can be seen as an attempt in this direction.

On the other hand, due to the divergence form of the equation \eqref{eq:s-LMA-div}, we already know that some interesting Sobolev inequalities of Monge-Amp\`ere type  were obtained by \cite{TW} in dimension $n\geq 3$ and \cite{Le} in dimension $n=2$ (see \cite{Ma} for some extensions and \cite{WZ23} for a complex version). Since we will use the Monge-Amp\`ere type Sobolev inequality later, we restate it here.
\begin{thm}[{\cite[Theorem 1.1]{TW}, \cite[Proposition 2.6]{Le}}]\label{lem:Sob}
Assume $n\geq 2$. Let $\phi\in C^2(\Omega)$ be a convex function satisfying 
\eqref{eq:det-bd}. Then there exists a constant $C_{Sob}>0$, depending only on $n$, $\lambda$, $\Lambda$,  and $\Omega$ such that
$$\left(\int_{\Omega}|u|^{2^*}\mathrm{d}x\right)^{\frac{1}{2^*}}\leq C_{Sob}\left(\int_{\Omega}\Phi^{ij}D_iuD_ju\mathrm{~d}x\right)^{\frac{1}{2}},\quad\forall\,u\in C_0^\infty(\Omega),$$
where $2^*=\frac{2n}{n-2}$ for $n\geq 3$, and  any $2^*> 2$ for $n=2$.
\end{thm}
Therefore, with the new idea in Theorem \ref{thm:main}, we also establish a new Moser-Trudinger type inequality in two dimensions. To simplify the notation, we write
$$\|Du\|_{\Phi}^2:=\int_{\Omega}\Phi^{ij}D_iuD_ju\mathrm{~d}x.$$
\begin{thm}\label{thm:MT-in}
    Let Ω be a uniformly convex domain in $\mathbb R^2$ and $\phi\in C^2(\Omega)$ be a convex function satisfying \eqref{eq:det-bd}. Assume that $\phi|_{\partial\Omega}$ and $\partial\Omega$ are of class $C^3$. For any $u\in C^\infty_0(\Omega)$, there exists a constant $C>0$ depending only on $\lambda$, $\Lambda$ $\|\phi\|_{C^3(\partial\Omega)}$, the uniform convexity radius of $\partial\Omega$ and the $C^3$ regularity of $\partial\Omega$ such that
    \begin{equation}\label{eq:MT-in}
        \int_{\Omega}e^{\beta\frac{u^2}{\|Du\|_{\Phi}^2}}\mathrm{~d}x_1\mathrm{d}x_2\leq C|\Omega|^{\frac{\varepsilon_0}{2+\varepsilon_0}},
    \end{equation}
    where $\beta\leq 4\pi\frac{1+\varepsilon_0}{2+\varepsilon_0}\min\{\lambda,1\}$, and $\varepsilon_0$ depending only on $\lambda$ and $\Lambda$ is obtained by the global $W^{2,1+\varepsilon}$-estimate for Monge-Amp\`ere equations.
\end{thm}
\begin{rem}
    If $\phi(x)=\frac{1}{2}|x|^2$, we know that $\{\Phi^{ij}\}=id$, $\lambda=1$ and $\varepsilon_0=+\infty$. So in this case the inequality \eqref{eq:MT-in} is identical to the classical Moser-Trudinger inequality.
\end{rem}

In higher dimensions, very little is known when the singular term $F$ appears. For general degenerate linear elliptic equations, there are some extensions of the classical De Giorgi-Nash-Moser theory with some integral conditions on the elliptic coefficients in \cite{MS,Tr}. 
Therefore, in the second part of this paper, 
 we will investigate the interior regularity in terms of \eqref{eq:s-LMA-div} in higher dimensions with certain assumptions. As mentioned earlier, directly applying the De Giorgi-Nash-Moser iteration faces challenges due to the lack of a suitable estimate for $|D^2\phi|$. Therefore, we need to introduce further assumptions on $F$ to allow the use of the De Giorgi-Nash-Moser iteration. These assumptions are specified in the following theorem.  Denote
$$S_\phi(x,h):=\{y\in\Omega\,|\,\phi(y)<\phi(x)+D\phi(x)\cdot(y-x)+h\}$$
as the section of $\phi$ centered at $x\in\Omega$ with height $h>0$.



\begin{thm}\label{thm:main-h}
    Let $\phi\in C^2(\Omega)$ be a convex function satisfying \eqref{eq:det-bd}. For $q>n$, let $F:=(F^1(x),\cdots,F^n(x)):\Omega\to\mathbb R^n$ be a vector field satisfying $\|F_\phi\|_{L^q(\Omega)}<\infty$, where $F_\phi(x):=\left(D^2\phi(x)\right)^{1/2}F$ and $f:\Omega\to\mathbb R$ be a function satisfying $\|f\|_{L^{q_*}(\Omega)}<\infty$, where $q_*=\frac{nq}{n+q}$. Given a section $S_\phi(x_0,2h_0)\subset\subset\Omega$, for every solution $u$ to \eqref{eq:s-LMA-div}
    in $S_\phi(x_0,2h_0)$ and for all $x\in S_\phi(x_0,h_0)$, there is
    \begin{equation*}
        |u(x)-u(x_0)|\leq C\left(\|u\|_{L^\infty(S_\phi(x_0,2h_0))}+\|F_\phi\|_{L^q(S_\phi(x_0,2h_0))}+\|f\|_{L^{q_*}(S_\phi(x_0,2h_0))}\right)|x-x_0|^\gamma,
    \end{equation*}
    where constant $\gamma>0$ depending only on $n$, $\lambda$ and $\Lambda$, and constant $C>0$ depending only on $n$, $q$, $\lambda$, $\Lambda$, $h_0$ and $\operatorname{diam}(\Omega)$.
\end{thm}
\begin{rem}
Some remarks in order.
    \begin{enumerate}


        \item With specific assumptions regarding the integral bound of $D^2\phi$, one can also get a result similar to Theorem \ref{thm:main-h} in all dimensions, see details in \cite[Theorem 15.6]{LeBook} and \cite[Corollary 1.2]{Kim}.

        \item It would be interesting to remove the assumption on $F_\phi$ and/or the integral bound of $D^2\phi$ in Theorem \ref{thm:main-h} and instead simply assume that \eqref{eq:det-bd} and $F$ is bounded. But so far we haven't come up with a way to deal with this when $n\geq 3$.
    \end{enumerate}
\end{rem}

The framework of the proof of Theorem \ref{thm:main-h} is similar to that of \cite{Le,Lo}. First, we need to derive a new weak maximum principle (Theorem \ref{thm:glob-est} and Corollary \ref{cor:glob-est}) for the solution of \eqref{eq:s-LMA-div} by De Giorgi's iteration. Combining the new weak maximum principle with the Caffarelli-Guti\'errez's Harnack inequality then gives Theorem \ref{thm:main-h}.

The rest of the paper is organized as follows. In Section \ref{sect:2d}, we first apply the partial Legendre transform to \eqref{eq:s-LMA-div}, then give the proof of Theorem \ref{thm:main}. The Moser-Trudinger type inequality is also proved in this section. Next, in Section \ref{sect:hd}, we present some estimates for linearized Monge-Amp\`ere equations and then proceed to prove Theorem \ref{thm:main-h}.

 \textbf{Acknowledgments.} The author would like to thank his PhD supervisor, Prof. Bin Zhou, for his constant encouragement and many helpful suggestions. In addition, the author also extends appreciation to Prof. Nam Q. Le for his generous guidance and valuable suggestions.

\vskip 20pt

\section{Linearized Monge-Amp\`ere equations in dimension two}\label{sect:2d}
\vskip 12pt

In this section, we present a new proof of the interior estimate for \eqref{eq:s-LMA-div} without Caffarelli-Guti\'errez's theory and establish a new Moser-Trudinger type inequality in dimension two.

\subsection{The new equation under partial Legendre transform}

In this subsection, we first derive the new equation under the partial Legendre transform. 
Let $\Omega\subset \mathbb R^2$ and $\phi(x_1, x_2)$ be a convex function on $\Omega$.
The partial Legendre transform in the $x_1$-variable is
\begin{equation}\label{eq:part-legendre}
\phi^{\star}(\xi, \eta)=\sup\{x_1 \xi-\phi(x_1, \eta)\},
\end{equation}
where the supremum is taken with respect to $x_1$ on the slice $\eta$ is the fixed constant, namely for all
$x_1$ such that $(x_1,\eta)\in\Omega$. This definition is taken from \cite{Liu}. Hence, when $\phi\in C^2(\Omega)$ is a strictly convex function, 
we will have a injective mapping $\mathcal{P}$ satisfying
\begin{equation}\label{eq:p-Le-tr}
(\xi, \eta)=\mathcal{P}(x_1, x_2):=\left(\phi_{x_1}, x_2\right) \in \mathcal{P}(\Omega):=\Omega^{\star},
\end{equation}
where $\phi_{x_1}:=D_{x_1}\phi$.
In this situation, we know that
$$\phi^{\star}(\xi, \eta)=x_1\phi_{x_1}(x_1,x_2)-\phi(x_1,x_2).$$
Indeed, it just needs $\phi$ to be strictly convex respect to $x_1$-variable \cite{GP}.
Then a direct calculation yields
$$
\frac{\partial(\xi, \eta)}{\partial(x_1, x_2)}=\left(\begin{array}{ccc}
\phi_{x_1 x_1} & & \phi_{x_1 x_2} \\[5pt]
0 & & 1
\end{array}\right), \quad \text { and } \quad \frac{\partial(x_1, x_2)}{\partial(\xi, \eta)}=\left(\begin{array}{ccc}
\frac{1}{\phi_{x_1 x_1}} & & -\frac{\phi_{x_1 x_2}}{\phi_{x_1 x_1}} \\[5pt]
0 & & 1
\end{array}\right).
$$
Hence,
\begin{eqnarray}
&& \phi_{\xi}^{\star}=x_1,\ \ \phi_{\eta}^{\star}=-\phi_{x_2}, \label{eq:part-le-d}
\\
&& \phi_{\xi \xi}^{\star}=\frac{1}{\phi_{x_1 x_1}},\ \  \phi_{\eta \eta}^{\star}=-\frac{\operatorname{det} D^{2} \phi}{\phi_{x_1 x_1}},\ \ \phi_{\xi \eta}^{\star}=-\frac{\phi_{x_1 x_2}}{\phi_{x_1 x_1}}.\label{eq:part-le-2}
\end{eqnarray}
Then we know that $\phi^\star$ is a solution to
\begin{equation}\label{eq:p-star}
    (\det D^2\phi)\phi_{\xi \xi}^{\star}+\phi_{\eta\eta}^{\star}=0.
\end{equation}
In order to derive the equation under the partial Legendre transform, we consider the associated functionals
of \eqref{eq:s-LMA-div}
\begin{equation}\label{eq:func}
    A(u):=\int_{\Omega}\Phi^{ij}D_iuD_ju-2F^iD_iu+2fu\mathrm{~d}x,
\end{equation}
where the repeated indices are summed. Denote $\widetilde{u}(\xi,\eta):=u(\phi_{\xi}^{\star},\eta)$, $\widetilde{F}(\xi,\eta):=F(\phi_{\xi}^{\star},\eta)$ and $\widetilde{f}(\xi,\eta):=f(\phi_{\xi}^{\star},\eta)$, then we have the following equation for $\widetilde{u}$.
\begin{prop}\label{prop:new-eq}
    Assume $n=2$. Let $u$ be a solution to \eqref{eq:s-LMA-div}, then $\widetilde{u}$ satisfies
    \begin{equation}\label{eq:plt-lma}
        \left(-\frac{\phi^{\star}_{\eta\eta}}{\phi^{\star}_{\xi\xi}}\widetilde{u}_{\xi}\right)_\xi+\widetilde{u}_{\eta\eta}=\left(\widetilde{F}^1-\widetilde{F}^2\phi^{\star}_{\xi\eta}\right)_\xi+\left(\widetilde{F}^2\phi^{\star}_{\xi\xi}\right)_\eta+\widetilde{f}\phi^{\star}_{\xi\xi}\quad\text{ in }\Omega^{\star}.
    \end{equation}
\end{prop}
\begin{proof}
    Note that in dimension two, the cofactor matrix $(\Phi_{ij})$ is
    $$\left(\begin{aligned}
        \phi_{x_2x_2}&&-\phi_{x_1x_2}\\[5pt] 
        -\phi_{x_1x_2}&&\phi_{x_1x_1}
    \end{aligned}\right),$$
    then \eqref{eq:func} becomes
    $$A(u)=\int_{\Omega}\phi_{x_2x_2}(u_{x_1})^2-2\phi_{x_1x_2}u_{x_1}u_{x_2}+\phi_{x_1x_1}(u_{x_2})^2-2F^1u_{x_1}-2F^2u_{x_2}+2fu\mathrm{~d}x.$$
    Note that
    \beq\label{eq:part-le-3}
    u_{x_1}=\frac{1}{\phi^{\star}_{\xi\xi}}\widetilde{u}_{\xi}, \ u_{x_2}=-\frac{\phi^{\star}_{\xi\eta}}{\phi^{\star}_{\xi\xi}}\widetilde{u}_{\xi}+\widetilde{u}_{\eta}, \ \mathrm{d}x_1\mathrm{d}x_2=\phi^{\star}_{\xi\xi}\mathrm{d}\xi\mathrm{d}\eta,
    \eeq 
    we have
    \begin{align*}
        &A(u)
        =\int_{\Omega^\star}\bigg[-\frac{\det D^2\phi^{\star}}{\phi^{\star}_{\xi\xi}}\left(\frac{\widetilde{u}_{\xi}}{\phi^{\star}_{\xi\xi}}\right)^2+2\frac{\phi^\star_{\xi\eta}}{\phi^{\star}_{\xi\xi}}\frac{\widetilde{u}_{\xi}}{\phi^{\star}_{\xi\xi}}\left(-\frac{\phi^{\star}_{\xi\eta}}{\phi^{\star}_{\xi\xi}}\widetilde{u}_{\xi}+\widetilde{u}_{\eta}\right)\\
        &\quad\quad\quad+\frac{1}{\phi^{\star}_{\xi\xi}}\left(-\frac{\phi^{\star}_{\xi\eta}}{\phi^{\star}_{\xi\xi}}\widetilde{u}_{\xi}+\widetilde{u}_{\eta}\right)^2-2\widetilde{F}^1\frac{\widetilde{u}_{\xi}}{\phi^{\star}_{\xi\xi}}-2\widetilde{F}^2\left(-\frac{\phi^{\star}_{\xi\eta}}{\phi^{\star}_{\xi\xi}}\widetilde{u}_{\xi}+\widetilde{u}_{\eta}\right)
        +2\widetilde{f}\widetilde{u}\bigg]\phi^{\star}_{\xi\xi}\mathrm{~d}\xi\mathrm{d}\eta\\[5pt]
        &=\int_{\Omega^\star}\Bigg[\left(-\frac{\det D^2\phi^{\star}}{{\phi^{\star}_{\xi\xi}}^2}-2\frac{{\phi^\star_{\xi\eta}}^2}{{\phi^{\star}_{\xi\xi}}^2}+\frac{{\phi^\star_{\xi\eta}}^2}{{\phi^{\star}_{\xi\xi}}^2}\right){\widetilde{u}_\xi}^2+2\frac{\phi^{\star}_{\xi\eta}}{\phi^{\star}_{\xi\xi}}\widetilde{u}_\xi\widetilde{u}_\eta-2\frac{\phi^{\star}_{\xi\eta}}{\phi^{\star}_{\xi\xi}}\widetilde{u}_\xi\widetilde{u}_\eta+{\widetilde{u}_\eta}^2\\
        &\quad\quad\quad-2\widetilde{F}^1\widetilde{u}_\xi+2\widetilde{F}^2\phi^\star_{\xi\eta}\widetilde{u}_\xi-2\widetilde{F}^2\phi^\star_{\xi\xi}\widetilde{u}_\eta+2\widetilde{f}\phi^{\star}_{\xi\xi}\widetilde{u}\Bigg]\mathrm{d}\xi\mathrm{d}\eta\\[5pt]
        &=\int_{\Omega^\star}\left(-\frac{\phi^\star_{\eta\eta}}{\phi^\star_{\xi\xi}}{\widetilde{u}_{\xi}}^2+{\widetilde{u}_\eta}^2-2\widetilde{F}^1\widetilde{u}_\xi+2\widetilde{F}^2\phi^\star_{\xi\eta}\widetilde{u}_\xi-2\widetilde{F}^2\phi^\star_{\xi\xi}\widetilde{u}_\eta+2\widetilde{f}\phi^{\star}_{\xi\xi}\widetilde{u}\right)\mathrm{d}\xi\mathrm{d}\eta
        =:A^{\star}(\widetilde{u}).
    \end{align*}
    Since $u$ is a critical point of the functional $A(u)$, we know that $\widetilde{u}$ is  a critical point of the functional $A^\star(\widetilde{u})$. Thus, it suffices to derive the Euler-Lagrange equation of $A^\star(\widetilde{u})$. See \cite{WZ} for the similar argument for the Monge-Amp\`ere type fourth order equation.

    For $\varphi\in C^\infty_0(\Omega^\star)$, by integration by parts, we have
    \begin{eqnarray*}
        &&\left.\frac{\mathrm{d} A^\star(\widetilde{u}+t\varphi)}{\mathrm{d}t}\right|_{t=0}\\[4pt]
        &=&2\int_{\Omega^\star}\bigg(-\frac{\phi^\star_{\eta\eta}}{\phi^\star_{\xi\xi}}{\widetilde{u}_{\xi}}\varphi_\xi+{\widetilde{u}_\eta}\varphi_\eta-\widetilde{F}^1\varphi_\xi
       +\widetilde{F}^2\phi^\star_{\xi\eta}\varphi_\xi-\widetilde{F}^2\phi^\star_{\xi\xi}\varphi_\eta+\widetilde{f}\phi^{\star}_{\xi\xi}\varphi\Bigg)\mathrm{d}\xi\mathrm{d}\eta\\[5pt]
        &=&2\int_{\Omega^\star}\bigg[-\left(-\frac{\phi^{\star}_{\eta\eta}}{\phi^{\star}_{\xi\xi}}\widetilde{u}_{\xi}\right)_\xi-\widetilde{u}_{\eta\eta}+\left(\widetilde{F}^1\right)_\xi
 -\left(\widetilde{F}^2\phi^{\star}_{\xi\eta}\right)_\xi+\left(\widetilde{F}^2\phi^{\star}_{\xi\xi}\right)_\eta+\widetilde{f}\phi^{\star}_{\xi\xi}\Bigg]\varphi\mathrm{~d}\xi\mathrm{d}\eta.
    \end{eqnarray*}
    Then we know that the Euler-Lagrange equation of $A^\star(\widetilde{u})$ is
    $$-\left(-\frac{\phi^{\star}_{\eta\eta}}{\phi^{\star}_{\xi\xi}}\widetilde{u}_{\xi}\right)_\xi-\widetilde{u}_{\eta\eta}+\left(\widetilde{F}^1\right)_\xi-\left(\widetilde{F}^2\phi^{\star}_{\xi\eta}\right)_\xi+\left(\widetilde{F}^2\phi^{\star}_{\xi\xi}\right)_\eta+\widetilde{f}\phi^{\star}_{\xi\xi}=0,$$
    which yields \eqref{eq:plt-lma}.
\end{proof}

\subsection{Proof of Theorem \ref{thm:main}}
In order to use the partial Legendre transform, we first recall the definition of modulus of convexity.
For a convex function $\phi$ on $\mathbb R^n$, the {\it modulus of convexity}, denoted by $m_\phi$, is defined by
\beq\label{mdc}
m_\phi(t):=\inf\{\phi(x)-\ell_z(x):|x-z|>t\},
\eeq
where $t>0$ and $\ell_z$ is a supporting function of $\phi$ at $z$. It is clear that $m_\phi$ must be a positive function for a strictly convex function. A result of Heinz \cite{He} implies that in two dimensions, if $\det D^2\phi\geq \lambda>0$, there exists a positive function $C(t)>0$ depending on $\lambda$ such that $m_\phi(t)\geq C(t)>0$ (for a more specific $C(t)$, see \cite[Lemma 2.5]{Liu}). Now for the partial Legendre transform, we consider the mapping
$$
(\xi, \eta)=\mathcal P(x_1, x_2)=(\phi_{x_1},x_2): B_R(0)\to \mathbb R^2.
$$
The following important property is revealed in \cite{Liu}.
\begin{lem}[{\cite[Lemma 2.1]{Liu}}]\label{mod}
There exists a constant $\delta>0$ depending on the modulus of convexity $m_\phi$ defined in \eqref{mdc}, such that $B_\delta(0)\subset \mathcal P(B_R(0))$.
\end{lem}
\begin{rem}
    Indeed, from the proof of {\cite[Lemma 2.1]{Liu}}, we can see that the dependence on $\delta$ only requires the lower bound of $m_{\phi}$.
\end{rem}

\begin{proof}[Proof of Theorem \ref{thm:main}]
    For any $x\in \Omega$, we denote $R=\frac{\operatorname{dist}(x,\partial\Omega)}{2}$. Without loss of generality, we assume $\mathcal P(x)=0$. Note that $\phi$ satisfies 
    \eqref{eq:det-bd}, i.e. $\lambda\leq \det D^2\phi\leq\Lambda$, hence we know that $m_\phi(R)\geq C(R)>0$. By Lemma \ref{mod}, there exists $\delta>0$ depending on $C(R)$ such that $B_\delta(0)\subset \mathcal P(B_R(x))$. According to Proposition \ref{prop:new-eq}, \eqref{eq:det-bd} and \eqref{eq:p-star}, we know that $\widetilde{u}$ satisfies \eqref{eq:plt-lma} in $B_\delta(0)$ with 
    $$0<\lambda\leq \det D^2\phi=-\frac{\phi^\star_{\eta\eta}}{\phi^\star_{\xi\xi}}\leq \Lambda.$$
    This means that \eqref{eq:plt-lma} is a uniformly elliptic equation in divergence form.
    
    By the $W^{2,1+\varepsilon}$-estimate of Monge-Amp\`ere equations \cite{DFS,Sc}, there exist $\varepsilon_0>0$ depending on $\lambda$, $\Lambda$, and $C_0>0$ depending on $R$, $\lambda$ and $\Lambda$ such that $$\|D^2\phi\|_{L^{1+\varepsilon_0}(B_R(x))}\leq C_0.$$ Hence, we have 
    \begin{align*}
        \int_{B_\delta(0)}(\phi^\star_{\xi\xi})^{2+\varepsilon_0}\mathrm{~d}\xi\mathrm{d}\eta&=\int_{\mathcal{P}^{-1}(B_\delta(0))}(\phi_{x_1x_1})^{-(2+\varepsilon_0)}\phi_{x_1x_1}\mathrm{~d}x_1\mathrm{d}x_2\\[5pt]
        &=\int_{\mathcal{P}^{-1}(B_\delta(0))}(\phi_{x_1x_1})^{-(1+\varepsilon_0)}\mathrm{~d}x_1\mathrm{d}x_2\\[5pt]
        &=\int_{\mathcal{P}^{-1}(B_\delta(0))}\left(\frac{\phi_{x_2x_2}}{\phi_{x_1x_1}\phi_{x_2x_2}}\right)^{1+\varepsilon_0}\mathrm{~d}x_1\mathrm{d}x_2\\[5pt]
        &\leq \lambda^{-(1+\varepsilon_0)}\int_{B_R(x)}(\phi_{x_2x_2})^{1+\varepsilon_0}\mathrm{~d}x_1\mathrm{d}x_2\\[5pt]
        &\leq C\lambda^{-(1+\varepsilon_0)}.
    \end{align*}
    Then by \eqref{eq:p-star}, we have
    \begin{align*}
        \int_{B_\delta(0)}(-\phi^\star_{\eta\eta})^{2+\varepsilon_0}\mathrm{~d}\xi\mathrm{d}\eta\leq \Lambda^{1+\varepsilon_0}\int_{B_\delta(0)}(\phi^\star_{\xi\xi})^{2+\varepsilon_0}\mathrm{~d}\xi\mathrm{d}\eta\leq C\Lambda^{1+\varepsilon_0}\lambda^{-(1+\varepsilon_0)}.
    \end{align*}
    Finally, the standard $W^{2,p}$ theory of uniformly elliptic equations yields $$\|\phi^\star_{\xi\eta}\|_{L^{2+\varepsilon_0}(B_\delta(0))}\leq C.$$
    With the assumptions on $F$ and $f$, we know that the right-hand sides of \eqref{eq:plt-lma} satisfy
    $$\|\widetilde{F}^1-\widetilde{F}^2\phi^{\star}_{\xi\eta}\|_{L^{2+\varepsilon_0}(B_\delta(0))}\leq C,\,\,\|\widetilde{F}^2\phi^\star_{\xi\xi}\|_{L^{2+\varepsilon_0}(B_\delta(0))}\leq C,$$
    and
    \begin{align*}
        \int_{B_\delta(0)}&\left|\widetilde{f}\phi^\star_{\xi\xi}\right|^{\frac{r(2+\varepsilon_0)}{1+\varepsilon_0+r}}\mathrm{d}\xi\mathrm{d}\eta\leq C\int_{\mathcal{P}^{-1}(B_\delta(0))}|f|^{\frac{r(2+\varepsilon_0)}{1+\varepsilon_0+r}}\left(\phi_{x_2x_2}\right)^{\frac{(1+\varepsilon_0)(r-1)}{1+\varepsilon_0+r}}\mathrm{d}x_1\mathrm{d}x_2\\
        &\leq C\left(\int_{B_R(x)}|f|^r\mathrm{d}x_1\mathrm{d}x_2\right)^{\frac{2+\varepsilon_0}{1+\varepsilon_0+r}}\left(\int_{B_R(x)}\left(\phi_{x_2x_2}\right)^{1+\varepsilon_0}\mathrm{d}x_1\mathrm{d}x_2\right)^{\frac{r-1}{1+\varepsilon_0+r}}\\
        &\leq C\|f\|_{L^r(B_R(x))}^{\frac{r(2+\varepsilon_0)}{1+\varepsilon_0+r}}.
    \end{align*}
    Note that $n=2$ and $2+\varepsilon_0>2$, $\frac{r(2+\varepsilon_0)}{1+\varepsilon_0+r}>1$ whenever $r>1$, then the De Giorgi-Nash-Moser's theory \cite[Theorem 8.24]{GT} (The original estimate in Theorem 8.24 of \cite{GT} holds in terms of $|\widetilde{u}|_{L^q(B_\delta(0))}$ for all $q>1$. However, by using analogous arguments as in \cite[Page 75]{HL}, we can extend the validity of this estimate to all $q>0$) yields
    $$\|\widetilde{u}\|_{C^\alpha(B_{\delta/2}(0))}\leq C\left(\|\widetilde{u}\|_{L^{\frac{p\varepsilon_0}{1+\varepsilon_0}}(B_\delta(0))}+k\right),$$
    where 
    \[k=\|\widetilde{F}^1-\widetilde{F}^2\phi^{\star}_{\xi\eta}\|_{L^{2+\varepsilon_0}(B_\delta(0))}+\|\widetilde{F}^2\phi^\star_{\xi\xi}\|_{L^{2+\varepsilon_0}(B_\delta(0))}+\|\widetilde{f}\phi^\star_{\xi\xi}\|_{L^{\frac{r(2+\varepsilon_0)}{1+\varepsilon_0+r}}(B_\delta(0))}.\]
    Note that by H\"older's inequality there is
    \begin{align*}
        \left(\int_{B_\delta(0)}\widetilde{u}^{\frac{p\varepsilon_0}{1+\varepsilon_0}}\mathrm{~d}\xi\mathrm{d}\eta\right)^{\frac{1+\varepsilon_0}{p\varepsilon_0}}&=\left(\int_{\mathcal{P}^{-1}(B_\delta(0))}u^{\frac{p\varepsilon_0}{1+\varepsilon_0}}\phi_{x_1x_1}\mathrm{d}x_1\mathrm{d}x_2\right)^{\frac{1+\varepsilon_0}{p\varepsilon_0}}\\[5pt]
        &\leq \|u\|_{L^p(B_R(x))}\cdot \|\phi_{x_1x_1}\|_{L^{1+\varepsilon_0}(B_R(x))}^{\frac{1+\varepsilon_0}{p\varepsilon_0}}\\[5pt]
        &\leq C\|u\|_{L^p(B_R(x))}.
    \end{align*}
    Hence, for the original function $u$, combining with the $C^{1,\alpha}$ estimate of Monge-Amp\`epre equation \cite{Ca} we know that there exists a $\gamma\in(0,1)$ such that
    $$\|u\|_{C^\gamma(\mathcal{P}^{-1}(B_{\delta/2}(0)))}\leq C\left(\|u\|_{L^p(B_R(x))}+\|F\|_{L^\infty(B_R(x))}+\|f\|_{L^r(B_R(x))}\right).$$
    By a standard covering argument (see for instance \cite[Remark 2.15]{FR}), we know that the estimate is true for any $\Omega'\subset\subset\Omega$, which completes the proof.
\end{proof}

\subsection{Proof of Moser-Trudinger type inequality} 
In this subsection, we provide the proof of Theorem \ref{thm:MT-in}.

Moser-Trudinger type inequalities find broad applications in the study of partial differential equations and geometric problems. The classical Moser-Trudinger inequality was initially derived by Trudinger \cite{Tr1}, using the power series expansion of the exponential function and Sobolev estimates for individual terms, while carefully examining the dependence on the exponent of the expansions. Subsequently, Moser \cite{Mo2} presented a more direct proof of this inequality and also determined the optimal exponent. 
Before giving the proof of Theorem \ref{thm:MT-in}, we first recall the classical Moser-Trudinger inequality.
\begin{thm}[{\cite[Theorem 1]{Mo2}}]\label{thm:c-MT-in}
    Let $u\in W^{1,n}_0(\Omega)$ for $n\geq 2$, and
    $$\int_{\Omega}|Du|^n\mathrm{~d}x\leq 1.$$
    Then there exists a constant $C$ which depends  only on $n$ such that
    $$\int_{\Omega}e^{\alpha u^p}\mathrm{d}x\leq C|\Omega|,$$
    where
    $$p=\frac{n}{n-1},\quad\alpha\leq\alpha_n:=n\omega_{n-1}^{\frac{1}{n-1}},$$
    and $\omega_{n-1}$ is $(n-1)$-dimensional surface measure of the unit sphere.
\end{thm}

\begin{proof}[Proof of Theorem \ref{thm:MT-in}] 
    By the global $W^{2,1+\varepsilon}$-estimate of Monge-Amp\`ere equations \cite{DFS,Sc,Sa} (with a detailed exposition available in \cite[Theorem 10.1]{LeBook}), we know that there exist $\varepsilon_0>0$ depending on $\lambda$ and $\Lambda$, and ${C_0}>0$ depending only on $\lambda$, $\Lambda$ $\|\phi\|_{C^3(\partial\Omega)}$, the uniform convexity radius of $\partial\Omega$ and the $C^3$ regularity of $\partial\Omega$ such that $$\|D^2\phi\|_{L^{1+\varepsilon_0}(\Omega)}\leq C_0.$$
    Then as the proof of Theorem \ref{thm:main}, we have
    $$\int_{\Omega^\star}(\phi^\star_{\xi\xi})^{2+\varepsilon_0}\mathrm{~d}\xi\mathrm{d}\eta\leq C.$$
     By \eqref{eq:det-bd} there is
     $$0<\lambda\leq \det D^2\phi=-\frac{\phi^\star_{\eta\eta}}{\phi^\star_{\xi\xi}}\leq \Lambda.$$
 Hence,  by \eqref{eq:part-le-2} and \eqref{eq:part-le-3}, we obtain     
     \begin{align*}
         \|Du\|_{\Phi}^2&=\int_{\Omega}\Phi^{ij}D_iuD_ju\mathrm{~d}x\\[5pt]&=
         \int_{\Omega}\phi_{x_2x_2}{u_{x_1}}^2-2\phi_{x_1x_2}u_{x_1}u_{x_2}+\phi_{x_1x_1}{u_{x_2}}^2\mathrm{~d}x_1\mathrm{d}x_2\\
         &=\int_{\Omega^\star}\left(-\frac{\phi^\star_{\eta\eta}}{\phi^\star_{\xi\xi}}{\widetilde{u}_{\xi}}^2+{\widetilde{u}_\eta}^2\right)\mathrm{d}\xi\mathrm{d}\eta\\[5pt]
         &\geq \min\{\lambda,1\}\int_{\Omega^\star}\left({\widetilde{u}_{\xi}}^2+{\widetilde{u}_\eta}^2\right)\mathrm{d}\xi\mathrm{d}\eta\\[5pt]
         &=\min\{\lambda,1\}\|D\widetilde{u}\|_{L^2(\Omega^\star)}^2.
     \end{align*}
     Then there is
     \begin{align*}
         &\quad\int_{\Omega}e^{\beta\frac{u^2}{\|Du\|_{\Phi}^2}}\mathrm{~d}x\leq\int_{\Omega^\star}e^{\beta\min\{\lambda,1\}^{-1}\frac{\widetilde{u}^2}{\|D\widetilde{u}\|_{L^2(\Omega^\star)}^2}}\phi^\star_{\xi\xi}\mathrm{~d}\xi\mathrm{d}\eta\\[5pt]
         &\leq\left(\int_{\Omega^\star}e^{\beta\min\{\lambda,1\}^{-1}\frac{2+\varepsilon_0}{1+\varepsilon_0}\frac{\widetilde{u}^2}{\|D\widetilde{u}\|_{L^2(\Omega^\star)}^2}}\mathrm{~d}\xi\mathrm{d}\eta\right)^{\frac{1+\varepsilon_0}{2+\varepsilon_0}}\left(\int_{\Omega^\star}(\phi^\star_{\xi\xi})^{2+\varepsilon_0}\mathrm{~d}\xi\mathrm{d}\eta\right)^{\frac{1}{2+\varepsilon_0}}.
     \end{align*}
      For any $\beta\leq 4\pi\frac{1+\varepsilon_0}{2+\varepsilon_0}\min\{\lambda,1\}$, by Theorem \ref{thm:c-MT-in} with $n=2$, we have
     $$\int_{\Omega^\star}e^{\beta\min\{\lambda,1\}^{-1}\frac{2+\varepsilon_0}{1+\varepsilon_0}\frac{\widetilde{u}^2}{\|D\widetilde{u}\|_{L^2(\Omega^\star)}^2}}\mathrm{~d}\xi\mathrm{d}\eta\leq C|\Omega^\star|.$$
     Then by H\"older's inequality, 
     \begin{align*}
         \quad\int_{\Omega}e^{\beta\frac{u^2}{\|Du\|_{\Phi}^2}}\mathrm{~d}x&\leq C|\Omega^\star|^{\frac{1+\varepsilon_0}{2+\varepsilon_0}}\\[5pt]
         &=C\left(\int_{\Omega}\phi_{x_1x_1}\mathrm{~d}x_1\mathrm{d}x_2\right)^{\frac{1+\varepsilon_0}{2+\varepsilon_0}}\\[5pt]
         &\leq C\left[|\Omega|^{\frac{\varepsilon_0}{1+\varepsilon_0}}\left(\int_{\Omega}(\phi_{x_1x_1})^{1+\varepsilon_0}\mathrm{~d}x_1\mathrm{d}x_2\right)^{\frac{1}{1+\varepsilon_0}}\right]^{\frac{1+\varepsilon_0}{2+\varepsilon_0}}\\[5pt]
         &\leq C|\Omega|^{\frac{\varepsilon_0}{2+\varepsilon_0}},
     \end{align*}
     which completes the proof.
\end{proof}

\vskip 20pt

\section{Linearized Monge-Amp\`ere equations in higher dimensions}\label{sect:hd}
\vskip 12pt

In this section, we establish the proof of Theorem \ref{thm:main-h} in two steps.
\subsection{Estimates for linearized Monge-Amp\`ere equations}
In this subsection, we prove a weak maximum principle for linearized Monge-Amp\`ere equations, which will be used in the proof of Theorem \ref{thm:main-h}.
\begin{thm}[Weak maximum principle]\label{thm:glob-est}
     Let $\phi\in C^2(\Omega)$ be a convex function satisfying \eqref{eq:det-bd}.   For $q>n$, let $F:\Omega\to\mathbb R^n$ be a vector field satisfying $\|F_\phi\|_{L^q(\Omega)}<\infty$, where $F_\phi(x):=\left(D^2\phi(x)\right)^{1/2}F$ and $f:\Omega\to\mathbb R$ be a function satisfying $\|f\|_{L^{q_*}(\Omega)}<\infty$, where $q_*=\frac{nq}{n+q}$. For every solution $u$ to
     \begin{equation}\label{eq:s-LMA-sub}
        D_j\left(\Phi^{ij}D_{i}u\right)\geq\operatorname{div} F+f\quad\text{in}\quad \Omega,
     \end{equation}
     we have
     $$\sup_{\Omega}u\leq \sup_{\partial \Omega}u^++C\left(n,q,\lambda,\Lambda,\operatorname{diam}(\Omega)\right)\left(\|F_\phi\|_{L^q(\Omega)}+\|f\|_{L^{q_*}(\Omega)}\right)|\Omega|^{\frac{1}{n}-\frac{1}{q}},$$
     where $u^+:=\max\{u,0\}$.
\end{thm}
As a corollary, we have
\begin{cor}[Global estimate for solutions to the Dirichlet problem]\label{cor:glob-est}
     Let $\phi\in C^2(\Omega)$ be a convex function satisfying \eqref{eq:det-bd}.  For $q>n$, let $F:\Omega\to\mathbb R^n$ be a vector field satisfying $\|F_\phi\|_{L^q(\Omega)}<\infty$, where $F_\phi(x):=\left(D^2\phi(x)\right)^{1/2}F$ and $f:\Omega\to\mathbb R$ be a function satisfying $\|f\|_{L^{q_*}(\Omega)}<\infty$, where $q_*=\frac{nq}{n+q}$. For every section $S_\phi(x_0,h)$ with $S_\phi(x_0,h_0)\subset\subset\Omega$ for $h_0\geq h$ and every solution $u$ to
     \begin{equation*}
         \left\{
         \begin{aligned}
             D_j\left(\Phi^{ij}D_{i}u\right)&=\operatorname{div} F+f&&\text{in}\quad S_\phi(x_0,h),\\
             u&=0&&\text{on}\quad \partial S_\phi(x_0,h),
         \end{aligned}
         \right.
     \end{equation*}
     we have
     $$\sup_{S_\phi(x_0,h)}|u|\leq C\left(n,q,\lambda,\Lambda,\operatorname{diam}(\Omega),h_0\right)\left(\|F_\phi\|_{L^q(S_\phi(x_0,h))}+\|f\|_{L^{q_*}(S_\phi(x_0,h))}\right)h^{\frac{1}{2}-\frac{n}{2q}}.$$
\end{cor}

We will use De Giorgi's iteration to prove Theorem \ref{thm:glob-est} and Corollary \ref{cor:glob-est}, and De Giorgi's iteration is a very powerful tool for dealing with elliptic equations in divergence form. It is usually reduced to the following iteration lemma, the proof of which can be found in \cite{CW}.
\begin{lem}[{\cite[Lemma 4.1]{CW}}]\label{lem:De-Gi}
    Let $\omega(t)$ be a nonnegative and nonincreasing function in an interval $[k_0,+\infty)$. Suppose that there holds for all $h>k\geq k_0$,
    \begin{equation*}\label{eq:it-ineq}
        \omega(h)\leq\frac{C}{(h-k)^\alpha}[\omega(k)]^\beta,
    \end{equation*}
    where $\alpha>0$ and $\beta>1$. Then we have
    $$\omega(k_0+d)=0,$$
    where
    \begin{equation*}
        d=C^{\frac{1}{\alpha}}[\omega(k_0)]^{\frac{\beta-1}{\alpha}}2^{\frac{\beta}{\beta-1}}.
    \end{equation*}
\end{lem}

First, we give the proof of Theorem \ref{thm:glob-est}, and it follows the same idea as the proof of \cite[Theorem 4.2]{CW}.
\begin{proof}[Proof of Theorem \ref{thm:glob-est}]
    Denote $l=\sup\limits_{\partial\Omega}u^+$. Consider $v=(u-k)^+$ for $k\geq l$. Note that $v=u-k$, $Dv=Du$ a.e. in $\{u>k\}$ and $v=0$, $Dv=0$ a.e. in $\{u\leq k\}$. Taking $v$ as test function in \eqref{eq:s-LMA-sub}, we have
    $$-\int_{\Omega}\Phi^{ij}D_iuD_jv\mathrm{~d}x\geq -\int_{\Omega}F_iD_iv\mathrm{~d}x+\int_{\Omega}fv\mathrm{~d}x.$$
    Then
    \begin{align*}
        \int_{\Omega}\Phi^{ij}D_ivD_jv\mathrm{~d}x&\leq\int_{\Omega}F_iD_iv\mathrm{~d}x-\int_{\Omega}fv\mathrm{~d}x\\[5pt]
        &=\int_{\Omega}\left(D^2\phi\right)^{1/2}F\cdot \left(D^2\phi\right)^{-1/2}Dv\mathrm{~d}x-\int_{\Omega}fv\mathrm{~d}x\\[5pt]
        &\leq \left(\|F_\phi\|_{L^q(\Omega)}\|D^\phi v\|_{L^2(\Omega)}+\|f\|_{L^{q_*}(\Omega)}\|v\|_{L^{2^*}(\Omega)}\right)|A(k)|^{\frac{1}{2}-\frac{1}{q}},
    \end{align*}
    where $q_*=\frac{nq}{n+q}$, $2^*=\frac{2n}{n-2}$ ( we only consider $n\geq 3$ here, and $n=2$ is similar), 
    $$D^\phi v:=\left(D^2\phi\right)^{-1/2}Dv\quad\text{and}\quad A(k)=\{x\in \Omega\,|\,u(x)>k\}.$$
    Note that by \eqref{eq:det-bd} and $\Phi=(\det D^2\phi)(D^2\phi)^{-1}$, there is
    $$\|D^\phi v\|_{L^2(\Omega)}=\left(\int_{\Omega}\left[( D^2\phi)^{-1}\right]^{ij}D_ivD_jv\mathrm{~d}x\right)^{\frac{1}{2}}\leq \lambda^{-1/2}\left(\int_{\Omega}\Phi^{ij}D_ivD_jv\mathrm{~d}x\right)^{\frac{1}{2}}.$$
    Denote
    $$F_0=\lambda^{-1/2}\|F_\phi\|_{L^q(\Omega)}+C_{Sob}\|f\|_{L^{q_*}(\Omega)},$$
    where $C_{Sob}$ is the Sobolev constant in Theorem \ref{lem:Sob}.
    Hence, we know by Cauchy-Schwarz's inequality and Theorem \ref{lem:Sob} that
    \begin{align*}
        \int_{\Omega}\Phi^{ij}D_ivD_jv\mathrm{~d}x&\leq \left(\|F_\phi\|_{L^q(\Omega)}\|D^\phi v\|_{L^2(\Omega)}+\|f\|_{L^{q_*}(\Omega)}\|v\|_{L^{2^*}(\Omega)}\right)|A(k)|^{\frac{1}{2}-\frac{1}{q}}\\[5pt]
        &\leq F_0\left(\int_{\Omega}\Phi^{ij}D_ivD_jv\mathrm{~d}x\right)^{\frac{1}{2}}|A(k)|^{\frac{1}{2}-\frac{1}{q}}\\[5pt]
        &\leq \frac{1}{2}\int_{\Omega}\Phi^{ij}D_ivD_jv\mathrm{~d}x+\frac{F_0^2}{2}|A(k)|^{1-\frac{2}{q}},
    \end{align*}
    which implies
    $$\int_{\Omega}\Phi^{ij}D_ivD_jv\mathrm{~d}x\leq F_0^2|A(k)|^{1-\frac{2}{q}},$$
    i.e.
    $$\left(\int_{\Omega}\Phi^{ij}D_ivD_jv\mathrm{~d}x\right)^{\frac{1}{2}}\leq F_0|A(k)|^{\frac{1}{2}-\frac{1}{q}}.$$
    By Theorem \ref{lem:Sob}, there is
    $$\|v\|_{L^{2^*}(\Omega)}\leq C_{Sob}F_0|A(k)|^{\frac{1}{2}-\frac{1}{q}},$$
    Note that $v=(u-k)^+$. Thus, when $h>k$, there is
    $$\|v\|_{L^{2^*}(\Omega)}\geq (h-k)|A(h)|^{\frac{1}{2^*}}.$$
    Hence, we have
    $$|A(h)|\leq\frac{\left(C_{Sob}F_0\right)^{2^*}}{(h-k)^{2^*}}|A(k)|^{\frac{n(q-2)}{q(n-2)}},\quad\forall\,h>k\geq l.$$
    Applying Lemma \ref{lem:De-Gi} with $\alpha=2^*$ and $\beta=\frac{n(q-2)}{q(n-2)}>1$ yields
    $$A(l+d)=0,$$
    where
    $$d=C_{Sob}F_0|A(0)|^{\frac{1}{n}-\frac{1}{q}}\cdot2^{\frac{n(q-2)}{2(q-n)}}.$$
    Therefore, we obtain that
    $$\sup_{\Omega}u\leq l+d\leq \sup_{\partial\Omega}u^++C\left(\|F_\phi\|_{L^q(\Omega)}+\|f\|_{L^{q_*}(\Omega)}\right)|\Omega|^{\frac{1}{n}-\frac{1}{q}}.$$
\end{proof}
Next, we prove Corollary \ref{cor:glob-est}.
\begin{proof}[Proof of Corollary \ref{cor:glob-est}]
    Using \eqref{eq:det-bd} and the volume estimates for sections (see, for example, \cite[Lemma 4.6]{Fi}), we obtain:
    $$c_1(n,\lambda,\Lambda)h^{n/2}\leq |S_\phi(x_0,h)|\leq C_1(n,\lambda,\Lambda)h^{n/2}$$
    Hence, applying Theorem \ref{thm:glob-est} with $\Omega=S_\phi(x_0,h)$ and $u=0$ on $\partial S_\phi(x_0,h)$, we have
    $$\sup_{S_\phi(x_0,h)}u\leq C\left(\|F_\phi\|_{L^q(S_\phi(x_0,h))}+\|f\|_{L^{q_*}(S_\phi(x_0,h))}\right)h^{\frac{1}{2}-\frac{n}{2q}}.$$
    For $-u$ we can similarly show
    $$\sup_{S_\phi(x_0,h)}(-u)\leq C\left(\|F_\phi\|_{L^q(S_\phi(x_0,h))}+\|f\|_{L^{q_*}(S_\phi(x_0,h))}\right)h^{\frac{1}{2}-\frac{n}{2q}}.$$
     Combining the two inequalities gives
    $$\sup_{S_\phi(x_0,h)}|u|\leq C\left(n,q,\lambda,\Lambda,\operatorname{diam}(\Omega),h_0\right)\left(\|F_\phi\|_{L^q(S_\phi(x_0,h))}+\|f\|_{L^{q_*}(S_\phi(x_0,h))}\right)h^{\frac{1}{2}-\frac{n}{2q}}.$$
\end{proof}

\subsection{Proof of Theorem \ref{thm:main-h}}
The proof of Theorem \ref{thm:main-h} is followed by the combination of Corollary \ref{cor:glob-est} with Caffarelli-Guti\'errez's Harnack inequality. For reader's convenient, we recall the Harnack inequality here.
\begin{thm}[{\cite[Theorem 5]{CG}}]\label{thm:harn}
    Assume $n\geq 2$. Let $\phi\in C^2(\Omega)$ be a convex function satisfying \eqref{eq:det-bd}. Let $u\in W^{2,n}_{loc}(\Omega)$ be a nonnegative solution of the homogeneous linearized Monge-Amp\`ere equation
    $$\Phi^{ij}D_{ij}u=0$$
    in a section $S_\phi(x_0,2h)\subset\subset\Omega$. Then there is
    \begin{equation*}
        \sup_{S_\phi(x_0,h)}u\leq C(n,\lambda,\Lambda)\inf_{S_\phi(x_0,h)}u.
    \end{equation*}
\end{thm}

The proof of Theorem \ref{thm:main-h} mirrors that of Theorem 1.3 in \cite[P284-P285]{Le}. For the sake of completeness, we outline it briefly here.
\begin{proof}[Sketch of the proof of Theorem \ref{thm:main-h}]
    By noticing the $C^{1,\alpha}$ of standard Monge-Amp\`ere equations provided \eqref{eq:det-bd}, we know it suffices to show for $h\leq h_0$, there is
    \begin{equation}\label{eq:pf-main-1}
        \operatornamewithlimits{osc}_{S_{\phi}(x_0,h)}u\leq C\left(\|u\|_{L^\infty(S_\phi(x_0,h))}+\|F_\phi\|_{L^q(S_\phi(x_0,h))}+\|f\|_{L^{q_*}(S_\phi(x_0,h))}\right)h^{\gamma_0},
    \end{equation}
    where $\gamma_0\in(0,1)$ depending only on $n$, $q$, $\lambda$ and $\Lambda$. To prove \eqref{eq:pf-main-1}, we split $u=v+w$ where
    $$
    \left\{
    \begin{aligned}
        \Phi^{ij}D_{ij}v&=\div F+f&&\text{in}\quad S_{\phi}(x_0,h),\\
        v&=0&&\text{on}\quad\partial S_{\phi}(x_0,h),
    \end{aligned}
    \right.$$
    and
    $$
    \left\{
    \begin{aligned}
        \Phi^{ij}D_{ij}w&=0&&\text{in}\quad S_{\phi}(x_0,h),\\
        w&=u&&\text{on}\quad\partial S_{\phi}(x_0,h).
    \end{aligned}
    \right.$$
    Then, by Corollary \ref{cor:glob-est},
    $$\sup_{S_{\phi}(x_0,h)}|v|\leq C\left(\|F_\phi\|_{L^q(S_\phi(x_0,h))}+\|f\|_{L^{q_*}(S_\phi(x_0,h))}\right)h^{\frac{1}{2}-\frac{n}{2q}}.$$
    By Caffarelli-Guti\'errez's Harnack inequality (Theorem \ref{thm:harn}),
    $$\operatornamewithlimits{osc}_{S_{\phi}(x_0,h/2)}w\leq \beta\operatornamewithlimits{osc}_{S_{\phi}(x_0,h)}w.$$
    Therefore, we have
    \begin{align*}
        \operatornamewithlimits{osc}_{S_{\phi}(x_0,h/2)}u&\leq\operatornamewithlimits{osc}_{S_{\phi}(x_0,h/2)}w+\operatornamewithlimits{osc}_{S_{\phi}(x_0,h/2)}v\\[5pt]
        &\leq \beta\operatornamewithlimits{osc}_{S_{\phi}(x_0,h)}w+2\|v\|_{L^\infty(S_{\phi}(x_0,h))}\\[5pt]
        &\leq \beta\operatornamewithlimits{osc}_{S_{\phi}(x_0,h)}u+C\left(\|F_\phi\|_{L^q(S_\phi(x_0,h))}+\|f\|_{L^{q_*}(S_\phi(x_0,h))}\right)h^{\frac{1}{2}-\frac{n}{2q}}.
    \end{align*}
    Finally, by a standard iteration, there is
    $$\operatornamewithlimits{osc}_{S_{\phi}(x_0,h)}u\leq C\left(\frac{h}{h_0}\right)^{\gamma_0}\left(\|u\|_{L^\infty(S_{\phi}(x_0,h))}+\left(\|F_\phi\|_{L^q(S_\phi(x_0,h))}+\|f\|_{L^{q_*}(S_\phi(x_0,h))}\right)h_0^{\frac{1}{2}-\frac{n}{2q}}\right)$$
    for some $\gamma_0\in(0,1)$ and some constant $C>0$ depending only on $n$, $q$, $\lambda$, $\Lambda$, $h_0$ and $\operatorname{diam}(\Omega)$, which yields \eqref{eq:pf-main-1}.
\end{proof}

\end{document}